\documentclass[11pts]{article}

\usepackage[margin=1in]{geometry}

\usepackage{amsmath,amssymb,latexsym,amsthm,graphicx}

\usepackage{mathrsfs}

\usepackage{url} 
\usepackage{wrapfig}
\usepackage{subfigure}
\usepackage{hyperref}
\usepackage[final]{pdfpages}

\newtheorem{theorem}{Theorem}[section]
\newtheorem{lemma}[theorem]{Lemma}

\newtheorem{defi}[theorem]{Definition}

\newcommand{\e}{\vec{e}}

\newcommand{\mM}{\mathcal{M}}

\newcommand{\sg}{\sigma}

\newcommand{\R}{\mathbb R}

\newcommand{\To}{\mathbf T}
\newcommand{\Ro}{\mathbf R}

\def\H{\mathbb{H}}

\newcommand{\mC}{\mathscr{C}}

\newcommand{\mc}{\mathbf c}

\vspace{-2cm}
\title{An Inversion Formula for Horizontal Conical Radon Transform}
\author{Duy N. Nguyen\footnote{High School for the Gifted, Ho Chi Minh City, Vietnam. Email: nnduy@ptnk.edu.vn.} ~ and Linh V. Nguyen\footnote{University of Idaho, 875 Perimeter Dr, Moscow, ID 83844, USA. Email: lnguyen@uidaho.edu.}}

\date{}
\begin{document}

\maketitle

\begin{abstract}
In this paper, we consider the conical Radon transform on all cones with horizontal central axis whose vertices are on a straight line. We derive an explicit inversion formula for such transform. The inversion makes use of the vertical slice transform on a sphere and V-line transform on a plane. 
\end{abstract} 

\section{Introduction}

Let us denote by $\mC$ the set all cones in $\R^n$. Then, a (weighted) conical Radon transform of a function $f \in C^\infty(\R^n)$ is the function $\To(f): \mM \subset \mC \to \R$ defined by 
$$\To(f)(\mc) = \int_\mc f(x) \, w(x,\mc)\,  d\sg(x), \quad c \in \mM,$$
where $w(x,\mc)$ is a positive smooth weight function. The conical Radon transform has been actively studied the thanks to its applications in Compton camera imaging (see, \cite{everett1977gamma,singh1983electronically}). In Compton camera imaging, one has to invert a conical Radon transform in order to find the interior image of a biological object from the measurement of Compton scattering.

In the two dimensional space ($n=2$), the conical Radon transform becomes the V-line transform, which also arises in optical tomography \cite{florescu2009single}. There exist quite a few inversion formulas for the V-line transform (e.g. \cite{basko1997analytical,morvidone2010v,truong2011new,florescu2011inversion,ambartsoumian2012inversion,hristova2015inversion}).  In the three dimensional space ($n=3$), $\mC$ is a six dimensional manifold and there are many practical choices of $\mM$. Taking advantage of redundancy, i.e. by choosing $dim(\mM) > 3$, was the topic of several works (see, e.g., \cite{terzioglu2015some,kuchment2016three}). One, however, may wish to study the case $dim(\mM) = 3$ for the mathematical interest and practical setups for Compton camera imaging \cite{cree1994towards,nguyen2005radon,allmaras2010detecting,gouia2014exact,moon2017analytic}. Several papers (e.g.,\cite{gouia2014analytical,haltmeier2014exact,terzioglu2015some}) gave the inversion formula for conical transform in general dimensional space. We mention that using spherical harmonics to compute series solutions is also a popular approach for inversion (see, e.g., \cite{basko1998application,jung2015inversion,schiefeneder2017radon}). 


\begin{wrapfigure}{R}{0.45\textwidth}
\centering
\includegraphics[width=0.45\textwidth]{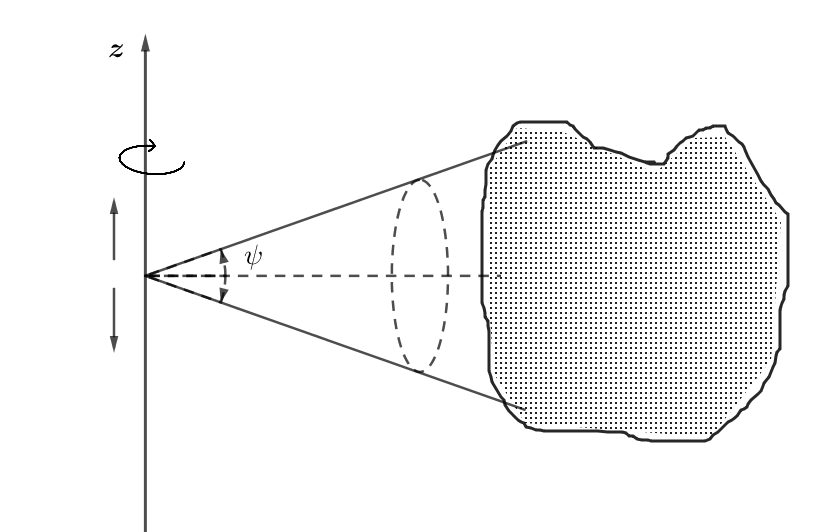}
\caption{\label{F:Ours} Our setups: the conical Radon transform over all cones with horizontal central line and the vertex in z-axis.}
\end{wrapfigure}

In this paper, we aim to reconstruct a function $f \in C^\infty(\R^3)$ from all cones whose vertices are in a vertical line and central lines are horizontal, see Fig.~\ref{F:Ours}. The manifold $\mM$ of such cones is of three dimensions. This formulation corresponds to the Compton camera imaging with detectors on a line. 


Let us now describe the problem in more detail. We introduce the following notations: $b\left(z\right)=\left( 0,0,z \right)\in \mathbb{R}^3$ be a point in the $z$-axis, ${{e}_{\varphi }}:=\left( \cos \varphi ,\sin \varphi ,0 \right)$ is a horizontal unit vector in the $xyz$-space, and $$\mc \left( z,\beta ,\psi  \right):=\left\{ b\left(z \right)+r\omega | r \geq 0 , \omega \in {{\mathbb S}^{2}},{{e}_{\beta }}\cdot\omega =\cos \psi  \right\}$$ is the one-side circular cone, having vertex $b\left( z \right)$ and the symmetry axis  $\left\{b(z) +r \cdot {{e}_{\beta }}|r>0 \right\}$.

We define our (weighted) conical Radon transform of a function $f\in C_{0}^{\infty }\left( \mathbb{R}^3\right)$ as follows.
\begin{eqnarray*}
\To_k(f) &:& \mathbb{R}\times \left[ 0,2\pi  \right)\times \left( 0,\frac{\pi }{2} \right) \longrightarrow \mathbb{R},\\
& & \left( z,\beta ,\psi  \right) \longmapsto \int\limits_{\mc \left( z,\beta ,\psi  \right)}{f\left( x \right){{\left\| x-b\left( z \right) \right\|}^{k-1}}dS\left( x \right)},
\end{eqnarray*} where $k \in \mathbb{N} $ is fixed. In this paper, we investigate the inversion of $\To_k$.

\section{The main results}

In order to invert the conical introduced in the previous section, we introduce the weighted X-ray and vertical slice transforms.

\begin{defi}[The weighted X-Ray transform] \label{D:con} Let $k\in \mathbb{N}$ and $f\in C_{0}^{\infty }\left( {\R^3} \right)$. We define the weighted X-Ray transform 
\begin{eqnarray*} {{\chi }_{k}}f &:& \mathbb{R}\times ({{\mathbb{R}}^{3}}\backslash \left\{ 0 \right\}) \longrightarrow \mathbb{R}, \\ \left( z,\omega  \right) &\longmapsto& \int\limits_{0}^{\infty }{f\left( b\left(z \right)+r\omega  \right){{r}^{k}}dr}.
\end{eqnarray*}
\end{defi} 
The weighted X-ray transform has been used in inverting the conical transform in other setups \cite{moon2017analytic}.

\begin{defi}[The vertical slice transform] Let $g\in C_{0}^{\infty }\left( \mathbb{S}^2 \right)$. We define the transform $\Gamma g:\left[ 0,2\pi  \right)\times \mathbb{R} \longrightarrow \mathbb{R}$ by the formula
\begin{eqnarray*} 
\Gamma g\left(\varphi ,t \right)= \left\{\begin{array}{l}\frac{1}{2\pi \sqrt{1-{{t}^{2}}}}\int\limits_{{{e}_{\varphi }} \cdot \omega=t}{g\left(\omega \right)\, d \omega},\mbox{ for } -1<t <1,\\[12 pt]
 g\left(\pm e_\varphi\right),\mbox{ for } t = \pm 1,\\[6 pt] 
0, \mbox{ for }|t|>1.
\end{array}
\right.
\end{eqnarray*}
\end{defi}

Vertical slice transform was investigated  Gindikin \cite{gindikin1994spherical}. It has the following inversion formula (see \cite[Theorem 2.1]{gindikin1994spherical}):

\begin{theorem} \label{T:Gin} 
 Let $\omega\in \mathbb{S}^2$ and $g : C(\mathbb{S}^2)\longrightarrow \mathbb{R} $ is even in the third coordinate. Then
  \begin{equation}
g \left(\omega \right) =  \dfrac{-\sqrt{1-\omega_1^2 - \omega_2^2}}{4\pi} \int\limits_{-\infty}^{+\infty} \dfrac{1}{t} \int\limits_0^{2\pi} \dfrac{\partial}{\partial t} \Gamma g \left( \varphi , \omega_1 \cos \varphi + \omega_2 \sin \varphi + t \right) \, d\varphi \, dt. 
  \end{equation}
\end{theorem}

The following lemma gives us the relationship between the weighted $X$-ray, the weighted conical Radon, and the vertical slice transforms:
\begin{lemma}\label{L:Gamma}
For every $k\in \mathbb{N}$ and $f\in C_{0}^{\infty }\left(\R^3\right)$, $z \in \mathbb{R}, \beta \in \left[0,2\pi\right), \psi \in \left( 0, \pi \right)$ then
\begin{equation} \label{E:GammaC}
	\Gamma \left( \chi _{k}f \right)\left(z,\beta,\cos \psi \right)=\dfrac{\To_k(f) \left( z,\beta ,\psi  \right)}{2\pi \sin^2 \psi }.
\end{equation}
\end{lemma}
In the lemma, we have used the notation $\Gamma (\chi_k f)(z, \cdot)$ for the vertical slice transform of $\chi_k f(z,\cdot)$. Let us now prove the lemma.

\begin{proof}
We have
\begin{align*}
 	 \To_k(f) \left( z,\beta ,\psi  \right) &=\int\limits_{\mc \left( \varphi ,z,\beta ,\psi  \right)}{f\left( x \right){{\left\| x-b\left(z \right) \right\|}^{k-1}}\, d\sigma\left( x \right)}   =\sin \psi \int\limits_{{\mathbb{S}^{2}}}{\int\limits_{0}^{\infty }{f\left( b\left(z \right)+r\omega  \right){{r}^{k}}\delta \left( {{e}_{\beta }}\cdot \omega -\cos \psi  \right) \,dr \,d\sigma\left( \omega  \right)}} \\
 	&=\sin \psi \int\limits_{{{\mathbb S}^{2}}}{{{\chi }_{k}}f\left( z,\omega  \right)\delta \left( {{e}_{\beta }}\cdot \omega -\cos \psi  \right)dS\left( \omega  \right)} = 2\pi {{\sin }^{2}}\left( \psi  \right)\Gamma \left( \chi_k f \right) \left( z,\beta ,\cos \psi  \right).
\end{align*}
This finishes the proof.
\end{proof}

	For any $\omega=\left(\omega_1,\omega_2,\omega_3 \right) \in \mathbb{R}^3$, we define the even and odd parts of the function $\chi_k f$ as follows:
\begin{eqnarray*}
(\chi_k f)_e\left(z,\omega\right) &=& \dfrac{\chi_k f\left(z,\omega_1,\omega_2,\omega_3\right)+ \chi_k f\left(z,\omega_1,\omega_2,-\omega_3\right)}{2},\\
(\chi_k f)_o\left(z,\omega\right) &=& \dfrac{\chi_k f\left(z,\omega_1,\omega_2,\omega_3\right)- \chi_k f\left(z,\omega_1,\omega_2,-\omega_3\right)}{2}.
\end{eqnarray*}
Notice that $\chi_k f\left( z,\omega  \right)=(\chi _k f)_e\left( z,\omega  \right)+ (\chi _k f)_o\left( z,\omega  \right)$. Since all the circles appearing in the vertical slice transform are symmetric with respect to the $xy$-plane and $(\chi _k f)_o$ is an odd function in the third coordinate, $\Gamma(\chi_k f_o) \equiv 0$. Therefore, from Lemma~\ref{L:Gamma}, we obtain
\begin{equation}\label{E:even} \Gamma (\chi_k f )_e \left( z,\beta ,\cos \psi  \right)= \Gamma (\chi_k f ) \left( z,\beta ,\cos \psi  \right) = \dfrac{\To_k(f)\left( z,\beta ,\psi  \right)}{2\pi \, {{\sin }^{2}} \psi }.\end{equation}
	 

Let $a\left( \varphi ,\eta  \right):=\left( \cos \varphi \sin \eta, \sin \varphi \sin \eta, \cos \eta  \right)\in {{\mathbb S}^{2}}$ be the unit vector in $\R^3$ determined by the horizontal angle $\varphi$ and vertical angle $\eta$. 
We denote $(\chi _k f)_e \left( z,\varphi , \eta \right) = (\chi _k f)_e \left( z, a\left( \varphi,\eta \right) \right) $, then:
\begin{lemma}\label{B:Gamma}
For $k \in \mathbb{N}$, $f\in C_0^\infty \left(\R^3\right)$, $z \in \mathbb{R}$, $\varphi \in \left[0, 2 \pi \right]$, $\eta \in \left (0, \pi \right)$,
\begin{align}
	(\chi_0  f)_{e} & \left( z,\varphi ,  \eta \right) = \dfrac{1}{8\pi ^2 (k-1)!} \int\limits_{\eta }^{\pi} \dfrac{\sin ^{k-1} \left( \gamma- \eta \right)}{\sin^k \eta} \left| \cos \gamma \right| \nonumber \\ 
	&\times \left\{ \int\limits_0^{2\pi} \left[ \int\limits_0^\pi \dfrac{1}{\cos \psi - \sin \gamma \cos\left( \beta - \varphi \right)  } \partial_z^k \left( \dfrac{\partial}{\partial \psi} \left( \dfrac{\To_k f\left( z, \beta , \psi \right) }{\sin ^2 \psi} \right)  \right) d \psi   \right]  d\beta  \right\} d \gamma. \label{F:Gamma}
\end{align}
\end{lemma}
\begin{proof}
Since $(\chi _k f)_e$ is even in the third coordinate, applying Lemma~\ref{L:Gamma} and Theorem \ref{T:Gin}, we obtain
\begin{align*}
(\chi _k f)_e \left( z, \omega \right) &=  \dfrac{-\sqrt{1-\omega_1^2 - \omega_2^2}}{4\pi} \int\limits_{-\infty}^{+\infty} \dfrac{1}{q} \int\limits_0^{2\pi} \dfrac{\partial}{\partial q} \Gamma\left( \chi_k f_{e} \right) \left( z, \beta , \omega_1 \cos \beta + \omega_2 \sin \beta + q \right) d\beta dq. 
\end{align*}
For any $\omega \in {\mathbb S}^2$, we can write $\omega = a(\varphi, \eta) = \left( \sin\eta \cos \varphi, \sin \eta \sin \varphi, \cos \eta \right)$. 	Therefore, 
\begin{align*}
(\chi _k f)_e \left( z, \omega \right)&= \dfrac{-\sqrt{1-\sin^2 \eta}}{4\pi} \int\limits_0^{2\pi} \int\limits_ {-\infty}^{+\infty} \dfrac{1}{q} \dfrac{\partial}{\partial q} \Gamma \left( \chi _k f_{e} \right) \left( z, \beta, \sin \eta \cos \left( \beta - \varphi \right) +q \right) dq d\beta.
	\end{align*}
Since $\Gamma (\chi_kf)_e\left(.,t\right)=0$ when $|t| >1$, we only need to consider when the third variable in the integrand is in the interval $[-1,1]$. We, hence, can define $\cos \psi = \sin \eta  \cos \left( \beta - \varphi \right) +q$ for some $\psi \in [0,\pi]$.  We arrive at
\begin{align*}
(\chi_k f)_{e} \left( z,\varphi, \eta \right) &= \dfrac{\left| \cos \eta \right|}{4\pi} \int\limits_0^{2\pi} \int\limits_0^\pi \dfrac{1}{\cos \psi - \sin \eta \cos \left( \beta - \varphi \right) } \dfrac{\partial}{\partial \psi} \Gamma \left( \chi_k f_{e} \right) \left( z, \beta, \cos \psi \right) d\psi d\beta.\end{align*}
Using (\ref{E:even}), we deduce
\begin{align*}
(\chi_k f)_{e} \left( z,\varphi, \eta \right) &= \dfrac{\left| \cos \eta \right|}{4\pi} \int\limits_0^{2\pi} \int\limits_0^\pi \dfrac{1}{\cos \psi - \sin \eta \cos \left( \beta - \varphi \right) } \dfrac{\partial}{\partial \psi} \left( \dfrac{\To_k(f)\left( z, \beta , \psi \right) }{2\pi \sin^2 \psi} \right) d\psi d\beta.
\end{align*}
On the other hand (see, e.g., \cite{moon2017analytic}), 
\begin{align*}
(\chi_0 f)_{e} \left( z, \omega_1, \omega_2, \omega_3 \right) &= \dfrac{1}{(k-1)!} \int\limits_{-\infty}^{\omega_3} \left( \omega_3 -s \right) ^{k-1} \partial _z^k \left( \chi_k f_{e}\right) \left( z, \omega_1, \omega_2, s \right) ds \\ &= \dfrac{1}{(k-1)!} \int\limits_{-\infty}^{\omega_3} \left( \omega_3-s \right) ^{k-1} \partial_z^k \left( \chi_k f_{e} \right) \left( z, \left( \dfrac{\omega_1, \omega_2, s }{h} \right) \right)  h^{-(k+1)} ds,
\end{align*}
for any $h > 0$.

Let us change the variable $s \to \gamma$ by the formula $s= \cot \gamma  \sin \eta, \, \gamma \in \left(0, \pi \right)$. Then, $ds=\dfrac{-\sin \eta}{\sin^2 \gamma } d \gamma$. Choosing $h=\dfrac{\sin \eta}{\sin \gamma}$, 
	\begin{align*}
	(\chi_0 f)_{e} \left( z, \varphi,\eta \right) = & \dfrac{1}{(k-1)!} \int\limits_{\eta } ^ {\pi } \left( \cos \eta - \sin \eta \cot \gamma \right) ^{k-1}  \partial_z^k \left( \chi_k f_{e} \right) \left( z, \left( \sin \gamma \theta \left( \varphi \right) , \cos \gamma \right) \right) \left( \dfrac{\sin \gamma}{\sin \eta} \right) ^{k+1} \dfrac{\sin \eta}{\sin^2 \gamma} d\gamma  \\ 
	= & \dfrac{1}{(k-1)!} \int\limits_{\eta } ^ {\pi } \dfrac{\sin^{k-1} \left(  \gamma - \eta \right) }{\sin^k \eta} \partial_z^k \left( \chi_k f_{e} \right) \left( z, \left(\sin \gamma \theta \left( \varphi \right) , \cos \gamma \right)\right) d\gamma \\ 
	= & \dfrac{1}{(k-1)!} \int\limits_{\eta } ^ {\pi } \dfrac{\sin^{k-1} \left( \gamma - \eta \right) }{\sin^k \eta} \\ 
	& \times  \partial_z^k \left\{ \dfrac{\left| \cos \gamma \right| }{4\pi} \int\limits_0^{2\pi} \int\limits_0^\pi \dfrac{1}{\cos \psi - \sin \gamma \cos \left( \beta - \varphi \right)}  \dfrac{\partial}{\partial \psi} \left( \dfrac{\To_k(f) \left( z, \beta, \psi \right)}{2\pi \sin^2 \psi} \right) d\psi d\beta \right\} d\gamma  \\ 
	= & \dfrac{1}{8\pi^2 (k-1)!} \int\limits_{\eta}^{\pi } \dfrac{\sin ^{k-1} \left( \gamma - \eta \right) }{\sin ^k \eta}  \left| \cos \gamma \right|  \\
	& \times  \left\{ \int\limits_0^{2\pi} \left[ \int\limits_0^\pi \dfrac{1}{\cos \psi - \sin \gamma \cos \left( \beta - \varphi \right) } \partial_z^k \left( \dfrac{\partial}{\partial \psi} \left( \dfrac{\To_k(f) \left( z, \beta, \psi \right) }{\sin^2 \psi} \right) \right) d\psi \right] d\beta \right\} d\gamma 
	\end{align*}
This finishes our proof. 
\end{proof}

Let us note $2 (X_0 f)_e(z,\varphi,\eta)$ is the integral of $f$ along a V-line whose vertex is $b(z)$, each branch makes an angle $\eta$ to the horizontal plane and is the reflection of the other via the horizontal plane. We are now ready to compute the function $f$ in $\R^3$ from its conical transform in Definition~\ref{D:con}. To this end, we will decompose $\R^3$ into the union of half-planes $\H_{\e}$, where $\e$ is a horizontal unit vector. Here, $\H_{\e}$ is the vertical half plane passing through the $z$-axis and containing the unit vector $\e$. We only need to compute $f$ on each such half-plane. On each such half-plane we are given the V-line transform of the function $f$, which integrates $f$ over all V-lines with horizontal central axis whose vertices are on the boundary of $\H_{\e}$. The inversion such transform can be reduced to that of the X-ray transform (see \cite{basko1997analytical}). The idea will be used in our proof below for Theorem~\ref{T:main} below.

\begin{theorem}[Inversion Formula of Conical Radon Transform] \label{T:main} For each $x \in \R^3$, we write $x= \left(r e_\varphi , x_3 \right) $ where $r \geq 0$ and $\varphi \in \left[0,2\pi \right)$. Then,
\begin{align*}
		f\left(x\right) &=\dfrac{1}{8\pi^4 (k-1)!}\int\limits_{0}^{\pi}\int_{\mathbb{R}}\dfrac{1}{r \cos\eta + x_3\sin\eta - p} \\
		&\times  \left\{ \int\limits_{\eta }^{\pi} \sin ^{k-1} \left( \gamma - \eta \right)|\cos \gamma | \int\limits_0^{2\pi} \left[ \int\limits_0^\pi \dfrac{\partial_p^{k+1}\left[\partial_{\psi} \left( \dfrac{\To_k(f) \left( p/\sin\eta, \beta , \psi \right) }{\sin ^2 \psi} \right)\right]}{\cos \psi - \sin \gamma \cos\left( \beta - \varphi \right)} d \psi   \right]  d\beta  \right\} \, d \gamma \, dp \, d\eta .\\
\label{G:Gamma}
\end{align*}

\end{theorem}	

\begin{proof} Let $\H$ be the vertical half-plane passing through the z-axis and the unit vector $ \vec{e} = e_\varphi$. For any $x \in \H$ then $x= \left(r \theta \left( \varphi \right) , x_3 \right) $.
We define $f^*\left( r ,x_3 \right) = f\left( r \theta \left( \varphi \right) , x_3 \right), \mbox{ for } r \geq 0$. We then extend $f^*$ to the whole space $(r,x_3) \in \R^2$ by the even reflection. Then, for any $\eta \in (0, \pi ), p \in \mathbb{R}$, we have
\begin{eqnarray*}
(\chi_0 f)_e\left( z,\varphi , \eta \right) = \dfrac{1}{2}\int\limits_{-\infty}^{+\infty} f^* \left( r \sin \eta , z-r\cos \eta \right)dr =\dfrac{1}{2} \Ro(f^*)\left(\eta,z\sin\eta\right).
\end{eqnarray*}
Here, $\Ro$ denote the standard Radon transform in two dimensional space. 
Choosing $z= \dfrac{p}{\sin \eta}$, then 
$$ \Ro (f^*) \left(\eta, p \right)= 2(\chi_0 f)_e\left( \dfrac{p}{\sin \eta} , \varphi,\eta \right), \mbox{ for } 0 < \eta < \pi.$$
Since $f \in C_0^\infty \left(\R^3\right)$, the value of $\Ro(f^*)$ at $\eta =0, \pi$ can be computed by the continuous extension. Now, using the inversion of the Radon transform  (see \cite{helgason1999radon}) and Lemma \ref{B:Gamma}, we conclude
\begin{align*}
		f^*\left(r,x_3\right)	&=\dfrac{1}{2\pi^2}\int\limits_{0}^{\pi}\int_{\mathbb{R}}\dfrac{\partial_p \Ro f^*\left( \eta, p \right)}{r\cos\eta + x_3\sin\eta - p}dp d\eta = \dfrac{1}{\pi^2}\int\limits_{0}^{\pi}\int_{\mathbb{R}}\dfrac{\partial_{p} \left( \chi_0 f_e \right) \left( \dfrac{p}{\sin \eta}, \varphi , \eta \right)}{r\cos\eta + x_3\sin\eta-p} dp d\eta \\
		&=\dfrac{1}{8\pi^4 (k-1)!}\int\limits_{0}^{\pi}\int_{\mathbb{R}}\dfrac{1}{ r\cos\eta + x_3\sin\eta - p} \\
		&\times  \left\{ \int\limits_{\eta }^{\pi} \sin ^{k-1} \left( \gamma - \eta \right)|\cos \gamma | \int\limits_0^{2\pi} \left[ \int\limits_0^\pi \dfrac{\partial_p^{k+1}\left[\partial_{\psi} \left( \dfrac{\To_k(f) \left( p/\sin\eta, \beta , \psi \right) }{\sin ^2 \psi} \right)\right]}{\cos \psi - \sin \gamma \cos\left( \beta - \varphi \right)} d \psi   \right]  d\beta  \right\} d \gamma dp d\eta. \\
\end{align*}
This finishes our proof.
\end{proof}	

\section*{Acknowlegement} Linh Nguyen's research is partially supported by the NSF grants, DMS 1212125 and DMS 1616904.


\begin{thebibliography}{10}

\bibitem{allmaras2010detecting}
Moritz Allmaras, David~P Darrow, Yulia Hristova, Guido Kanschat, and Peter
  Kuchment.
\newblock Detecting small low emission radiating sources.
\newblock {\em arXiv preprint arXiv:1012.3373}, 2010.

\bibitem{ambartsoumian2012inversion}
Gaik Ambartsoumian.
\newblock Inversion of the v-line radon transform in a disc and its
  applications in imaging.
\newblock {\em Computers \& Mathematics with Applications}, 64(3):260--265,
  2012.

\bibitem{basko1997analytical}
Roman Basko, Gengsheng~L Zeng, and Grant~T Gullberg.
\newblock Analytical reconstruction formula for one-dimensional compton camera.
\newblock {\em IEEE Transactions on Nuclear Science}, 44(3):1342--1346, 1997.

\bibitem{basko1998application}
Roman Basko, Gengsheng~L Zeng, and Grant~T Gullberg.
\newblock Application of spherical harmonics to image reconstruction for the
  compton camera.
\newblock {\em Physics in Medicine \& Biology}, 43(4):887, 1998.

\bibitem{cree1994towards}
Michael~J Cree and Philip~J Bones.
\newblock Towards direct reconstruction from a gamma camera based on compton
  scattering.
\newblock {\em IEEE transactions on medical imaging}, 13(2):398--407, 1994.

\bibitem{everett1977gamma}
DB~Everett, JS~Fleming, RW~Todd, and JM~Nightingale.
\newblock Gamma-radiation imaging system based on the compton effect.
\newblock In {\em Proceedings of the Institution of Electrical Engineers},
  volume 124, pages 995--1000. IET, 1977.

\bibitem{florescu2011inversion}
Lucia Florescu, Vadim~A Markel, and John~C Schotland.
\newblock Inversion formulas for the broken-ray radon transform.
\newblock {\em Inverse Problems}, 27(2):025002, 2011.

\bibitem{florescu2009single}
Lucia Florescu, John~C Schotland, and Vadim~A Markel.
\newblock Single-scattering optical tomography.
\newblock {\em Physical Review E}, 79(3):036607, 2009.

\bibitem{gindikin1994spherical}
S~Gindikin, J~Reeds, and L~Shepp.
\newblock Spherical tomography and spherical integral geometry.
\newblock {\em Tomography, Impedance Imaging, and Integral Geometry (South
  Hadley, MA, 1993), Lectures in Appl. Math}, 30:83--92, 1994.

\bibitem{gouia2014analytical}
Rim Gouia-Zarrad.
\newblock Analytical reconstruction formula for n-dimensional conical radon
  transform.
\newblock {\em Computers \& Mathematics with Applications}, 68(9):1016--1023,
  2014.

\bibitem{gouia2014exact}
Rim Gouia-Zarrad and Gaik Ambartsoumian.
\newblock Exact inversion of the conical radon transform with a fixed opening
  angle.
\newblock {\em Inverse Problems}, 30(4):045007, 2014.

\bibitem{haltmeier2014exact}
Markus Haltmeier.
\newblock Exact reconstruction formulas for a radon transform over cones.
\newblock {\em Inverse Problems}, 30(3):035001, 2014.

\bibitem{helgason1999radon}
Sigurdur Helgason and S~Helgason.
\newblock {\em The radon transform}, volume~2.
\newblock Springer, 1999.

\bibitem{hristova2015inversion}
Yulia Hristova.
\newblock Inversion of a v-line transform arising in emission tomography.
\newblock {\em Journal of Coupled Systems and Multiscale Dynamics},
  3(3):272--277, 2015.

\bibitem{jung2015inversion}
Chang-Yeol Jung and Sunghwan Moon.
\newblock Inversion formulas for cone transforms arising in application of
  compton cameras.
\newblock {\em Inverse Problems}, 31(1):015006, 2015.

\bibitem{kuchment2016three}
Peter Kuchment and Fatma Terzioglu.
\newblock Three-dimensional image reconstruction from compton camera data.
\newblock {\em SIAM Journal on Imaging Sciences}, 9(4):1708--1725, 2016.

\bibitem{moon2017analytic}
Sunghwan Moon and Markus Haltmeier.
\newblock Analytic inversion of a conical radon transform arising in
  application of compton cameras on the cylinder.
\newblock {\em SIAM Journal on imaging sciences}, 10(2):535--557, 2017.

\bibitem{morvidone2010v}
Marcela Morvidone, Ma{\"\i}~Khuong Nguyen, Tuong~T Truong, and Habib Zaidi.
\newblock On the v-line radon transform and its imaging applications.
\newblock {\em Journal of Biomedical Imaging}, 2010:11, 2010.

\bibitem{nguyen2005radon}
Mai~K Nguyen, Tuong~T Truong, and Pierre Grangeat.
\newblock Radon transforms on a class of cones with fixed axis direction.
\newblock {\em Journal of Physics A: Mathematical and General}, 38(37):8003,
  2005.

\bibitem{schiefeneder2017radon}
Daniela Schiefeneder and Markus Haltmeier.
\newblock The radon transform over cones with vertices on the sphere and
  orthogonal axes.
\newblock {\em SIAM Journal on Applied Mathematics}, 77(4):1335--1351, 2017.

\bibitem{singh1983electronically}
Manbir Singh.
\newblock An electronically collimated gamma camera for single photon emission
  computed tomography. part i: Theoretical considerations and design criteria.
\newblock {\em Medical Physics}, 10(4):421--427, 1983.

\bibitem{terzioglu2015some}
Fatma Terzioglu.
\newblock Some inversion formulas for the cone transform.
\newblock {\em Inverse Problems}, 31(11):115010, 2015.

\bibitem{truong2011new}
Tuong~T Truong and Mai~K Nguyen.
\newblock On new v-line radon transforms in $\mathbb{R}^2$ and their inversion.
\newblock {\em Journal of Physics A: Mathematical and Theoretical},
  44(7):075206, 2011.

\end{thebibliography}
\end{document}